\documentclass[12pt,a4paper]{amsart}

\usepackage{amsmath,amsfonts,amssymb,amsthm}
\usepackage{mathtools,mathdots}
 
\theoremstyle{plain}

\newtheorem{thm}{Theorem}[section]
\newtheorem{cor}[thm]{Corollary}
\newtheorem{lem}[thm]{Lemma}
\newtheorem{prop}[thm]{Proposition}

\newtheorem{thmABC}{Theorem}

\newtheorem{corABC}[thmABC]{Corollary}

\theoremstyle{definition}

\newtheorem{defn}[thm]{Definition}

\newtheorem*{defn*}{Definition}
\newtheorem*{ackn*}{Acknowledgement}

\numberwithin{equation}{section}

\newcommand{\norm}[1]{{\lvert #1 \rvert}}
\newcommand{\wt}[1]{\widetilde{#1}}
\newcommand{\wh}[1]{\widehat{#1}}

\newcommand{\cwr}{\mbox{\textnormal{\small\textcircled{$\wr$}}}}

\newcommand{\mP}{\mathcal{P}}
\newcommand{\mS}{\mathcal{S}}

\newcommand{\mX}{\mathcal{X}}

\newcommand{\Sym}{\mathrm{Sym}}

\title[Hereditarily just infinite profinite wreath products]{Embedding
  properties of hereditarily just infinite profinite wreath products}

\author{Benjamin Klopsch} \address{Mathematisches Institut der
  Heinrich-Heine-Universit\"at, Universit\"atsstr.\ 1, 40225
  D\"usseldorf, Germany}\email{klopsch@math.uni-duesseldorf.de}

\author{Matteo Vannacci} \address{Mathematisches Institut der
  Heinrich-Heine-Universit\"at, Universit\"atsstr.\ 1, 40225
  D\"usseldorf, Germany}\email{matteo.vannacci@uni-duesseldorf.de}



\begin{document} 

\maketitle


\begin{abstract}
  We study infinitely iterated wreath products of finite permutation
  groups w.r.t.\ product actions. In particular, we prove that, for
  every non-empty class of finite simple groups $\mathcal{X}$, there
  exists a finitely generated hereditarily just infinite profinite
  group $W$ with composition factors in $\mathcal{X}$ such that any
  countably based profinite group with composition factors in
  $\mathcal{X}$ can be embedded into $W$. Additionally we investigate
  when infinitely iterated wreath products of finite simple groups
  w.r.t.\ product actions are co-Hopfian or non-co-Hopfian.
\end{abstract}

\section{Introduction and main results} 

\subsection{Introduction} A profinite group $G$ is just infinite if
$G$ is infinite and every non-trivial closed normal subgroup
$N \trianglelefteq_\mathrm{c} G$ is open in~$G$.  While a complete
classification of just infinite profinite groups is way out of reach,
there is a natural interest in understanding as much about their
structure as possible.  It is known (e.g.,
see~\cite[Theorem~3]{grigorchuk:justinfinitebranch}) that every just
infinite profinite group either is a profinite branch group or
contains an open subgroup isomorphic to the direct product of a finite
number of copies of a hereditarily just infinite profinite group,
where a profinite group $G$ is called \emph{hereditarily just
  infinite} if every open subgroup $H \le_\mathrm{o} G$ is just
infinite.  While branch groups have been studied quite extensively
(e.g., see~\cite{bgs:branchgroups}) comparatively little is known
about hereditarily just infinite groups.

Well-known families of hereditarily just infinite profinite groups are
supplied by compact open subgroups of simple algebraic groups over
non-archimedean local fields, e.g., groups such as
$\mathrm{SL}_n(\mathbb{Z}_p)$ or
$\mathrm{SL}_n(\mathbb{F}_p[\![t]\!])$; see~\cite{linear prop}.  In
addition there are some `sporadic' non-linear examples, such as
$\mathrm{Aut}(\mathbb{F}_p[\![t]\!])$ and certain subgroups thereof;
see~\cite{camina:nottinghamgroup,bk:nottinghamgroup,MR2047455}. In
\cite[Theorem~A]{wilson:largehereditarily}, J.~S.~Wilson gave the
first examples of hereditarily just infinite profinite groups that are
not virtually pro-$p$ for any prime~$p$.  They arise as certain
iterated wreath products of non-abelian finite simple groups, and
retrospectively the construction is very flexible.  In
\cite{wilson:largehereditarily,quick:probabilisticgeneration,MR3466595}
some embedding, generation and presentation properties of such groups
have been established, but many of their features are not yet fully
understood.  In passing, we remark that A.~Lucchini has used
crown-based powers to manufacture further examples of hereditarily
just infinite profinite groups;
see~\cite{lucchini:a2generated}. Interesting new types of hereditarily
just infinite pro-$p$ groups were constructed by Ershov and Jaikin
in~\cite{ej:positivedeficiency}.

In this paper we focus on hereditarily just infinite profinite groups
that are obtained as inverse limits of iterated wreath products
w.r.t.\ product actions.  They arise as follows; see
Section~\ref{sec:prelim} for a more detailed description.  Let
$\mS = (S_k)_{k\in \mathbb{N} \cup \{0\}}$, with $S_k \le \Sym(\Omega_k)$, be a
sequence of finite transitive permutation groups.  The inverse limit
\[
W^\mathrm{pa}(\mS) = \varprojlim W^\mathrm{pa}_n
\]
of the inverse system
$W^\mathrm{pa}_0 \twoheadleftarrow W^\mathrm{pa}_1 \twoheadleftarrow
\ldots$ of finite iterated wreath products w.r.t.\ product actions
\[
  W^\mathrm{pa}_n  = S_n \,\cwr\, (S_{n-1} \,\cwr\, ( \cdots \,\cwr\,
  S_0 )) \le
                    \Sym(\wh{\Omega}_n) \quad \text{for $\wh{\Omega}_n =
                      \Omega_n^{\big(\Omega_{n-1}^{\big(\iddots^{\Omega_0}\big)} \big)}$}. 
\]
is called the \emph{infinitely iterated wreath product of type $\mS$
  w.r.t.\ product actions}.

By \cite[Theorem~6.2]{reid:characterization} and
\cite{quick:probabilisticgeneration}, every infinitely iterated
wreath product w.r.t.\ product actions $W^\mathrm{pa}(\mS)$, based on a
sequence $\mS$ of finite non-abelian simple permutation groups, is a
finitely generated hereditarily just infinite profinite group that is
not virtually pro-$p$ for any prime~$p$.  

\subsection{Main results} The aim of this paper is to study embedding
properties of infinitely iterated wreath product of finite non-abelian
simple groups w.r.t.\ product actions.  Specifically, we are interested
in embeddings of countably based profinite groups with specified
(topological) composition factors into such wreath products.
By~\cite[Theorem A]{wilson:largehereditarily} and
\cite{quick:probabilisticgeneration}, there exists a finitely
generated hereditarily just infinite profinite group $\mathcal{G}$
such that every countably based profinite group can be embedded into
$\mathcal{G}$ as a closed subgroup.  Our first theorem is a refinement
of this result to profinite groups with restricted composition
factors.  Recall that, by virtue of the Jordan--H\"{o}lder Theorem for
finite groups, every countably based profinite group $G$ has a
countable set of composition factors with well-defined multiplicities;
cf.\ Section~\ref{sec:prelim}.

\begin{thmABC}\label{thm:A}
  Let $\mS = (S_n )_{n\in \mathbb{N} \cup \{0\}}$ be a sequence of
  finite simple groups.  Then every profinite group $G$ that admits a
  composition series
  \[
  G = G_1 \triangleright G_2 \triangleright \ldots \quad  \text{with
    factors} \quad
  G_k/G_{k+1} \cong S_k, \, k \in \mathbb{N},
  \]
  embeds as a closed subgroup into the infinitely iterated wreath
  product $W^\mathrm{pa}(\mS)$ of type
  $\mS = ( S_k )_{k \in \mathbb{N} \cup \{ 0\}}$ w.r.t.\ product
  actions, where each $S_k \le \Sym(S_k)$ acts regularly on itself by
  right multiplication.
\end{thmABC}

We emphasise that Theorem~\ref{thm:A} includes the possibility of some
of the simple groups $S_k$ being cyclic.  Making further adjustments,
we construct for any given class of finite simple groups $\mX$ an
infinitely iterated wreath product $W^\mathrm{pa}(\mS_\mX)$ with
composition factors in $\mX$ that satisfies a `universal property' for
embedding countably based profinite group with composition factors
in~$\mX$.  The construction is flexible and the resulting group is in
general not unique.

\begin{corABC}\label{cor:B}
  Let $\mX$ be a non-empty class of finite simple groups.  Then there
  exists a sequence $\mS_\mX = (S_k)_{k \in \mathbb{N}}$ of groups
  $S_k \in \mX$, where each $S_k \le \Sym(S_k)$ acts regularly on itself
  by right multiplication, such that every countably based profinite
  group with composition factors in $\mathcal{X}$ embeds as a closed
  subgroup into the infinitely iterated wreath product
  $W^\mathrm{pa}(\mS_\mX)$ of type~$\mS_\mX$.
\end{corABC}

Our proof of Theorem~\ref{thm:A} relies on an apparently little known
construction to embed iterated wreath products w.r.t.\ imprimitive
actions into iterated wreath products w.r.t.\ product actions; see
Proposition~\ref {prop:P-embed} and its
Corollary~\ref{cor:finite-iterated-wrp}.

Furthermore we are interested in when infinitely iterated wreath
products of finite simple groups w.r.t.\ product actions are or fail
to be \emph{co-Hopfian}.  Recall that a profinite group $G$ is
\emph{co-Hopfian}, if there exists no proper closed subgroup
$H \lneqq_\mathrm{c} G$ with $H \cong G$.  Our description of
non-co-Hopfian groups relies on the concept of `permutational
isomorphism' of permutation groups; compare
\cite[p.~17]{dixon:permutationgroups} and see
Definition~\ref{defn:Pembedded} for a natural generalisation.

\begin{defn*} We say that a permutation group $H \le \Sym(\Delta)$ is
  \emph{permutationally isomorphic to a subgroup} of a permutation group
  $G \le \Sym(\Omega)$ if there exist $\wt{H} \le G$ and an
  $\wt{H}$-invariant subset $\wt{\Delta} \subseteq \Omega$ such that
  $H \le \Sym(\Delta)$ is equivalent to the faithfully induced
  permutation group
  $\wt{H} \vert_{\wt{\Delta}} \le \Sym(\wt{\Delta})$: there exist a
  group isomorphism $\iota \colon H \to \wt{H}$ and a bijection
  $\gamma \colon \Delta \to \wt{\Delta}$ such that
  $\gamma(\delta^h) = \gamma(\delta)^{\iota(h)}$ for all
  $\delta \in \Delta$ and $h \in H$. For instance, the permutation
  group $H = \langle (1\ 2) \rangle \le \Sym(2)$ is permutationally
  isomorphic to a subgroup of
  $G = \langle (1\ 2)(3\ 4)\rangle \le \Sym(4)$ via $\wt{H} = G$ and
  $\wt{\Delta} = \{3,4\}$.

  Observe that the relation ``permutationally isomorphic to a subgroup''
  on permutation groups is transitive.  We say that the terms of a
  sequence $\mS = (S_k)_{k\in \mathbb{N}}$ of finite permutation
  groups $S_k \le \Sym(\Omega_k)$ are \emph{eventually permutationally
    isomorphic to subgroups of later terms}, if there exists
  $n_0 \in \mathbb{N}$ such that, for every
  $j \in \mathbb{N}_{\ge n_0}$, there is at least one (equivalently:
  there are infinitely many) $k \in \mathbb{N}_{>j}$ for which
  $S_j \le \Sym(\Omega_j)$ is permutationally isomorphic to a subgroup of
  $S_k \le \Sym(\Omega_k)$.
\end{defn*}

We establish the following results.  

\begin{thmABC}\label{thm:C}
  Let $\mS = (S_k)_{k\in \mathbb{N}}$ be a sequence of non-trivial
  finite permutation groups $S_k \le \Sym(\Omega_k)$.  If the terms of
  $\mS$ are eventually permutationally isomorphic to subgroups of later
  terms, then the infinitely iterated wreath product
  $W^\mathrm{pa}(\mS)$ of type $\mS$ is non-co-Hopfian.
\end{thmABC}

Theorem~\ref{thm:C} applies, in particular, to constant sequences of
finite simple groups, but also to the sequence
$\mS = (S_k)_{k \in \mathbb{N}}$ of pairwise non-isomorphic
alternating groups $S_k = \mathrm{Alt}(k+4) \le \Sym(k+4)$.

For our final result, recall that a finite group $S$ is
\emph{minimal non-abelian simple} if it is non-abelian simple and
every proper subgroup of $S$ is soluble; such groups were classified
by J.~G.~Thompson~\cite{thompson:a,thompson:d} well before the
classification of all finite simple groups.

\begin{corABC} \label{cor:D} Let $\mS = (S_k)_{k\in \mathbb{N}}$ be a
  sequence of finite transitive permutation groups
  $S_k\le \Sym(\Omega_k)$ that are minimal non-abelian simple. Then
  the infinitely iterated wreath product $W^\mathrm{pa}(\mS)$ is
  non-co-Hopfian if and only if the terms of $\mS$ are eventually
  permutationally isomorphic to subgroups of later terms.
\end{corABC}

Observe that, if $\mS$ consists of minimal non-abelian simple groups,
then the terms of $\mS$ are eventually permutationally isomorphic to
subgroups of later terms if and only if almost all terms occur
infinitely often in~$\mS$.

To build an explicit example, we recall that the minimal finite
non-abelian simple groups are: $\mathrm{PSL}_2(2^p)$ for any prime
$p$, $\mathrm{PSL}_2(3^p)$ for any odd prime $p$, $\mathrm{PSL}_2(p)$
where $p>3$ and $5$ divides $p^2+1$, $\mathrm{Sz}(2^p)$ for any odd
prime $p$ and $\mathrm{PSL}_3(3)$.  Let $(p_k)_ {k \in \mathbb{N}}$ be
any sequence of prime numbers without repeated terms. Then
Corollary~\ref{cor:D} yields that the infinitely iterated wreath
product w.r.t. product actions of type
$\mS = (S_k)_{k \in \mathbb{N}}$, with
$S_k = \mathrm{PSL}_2(2^{p_k})\le \Sym(2^{p_k}+1)$ acting on the
projective line over $\mathbb{F}_{2^{p_k}}$, is co-Hopfian.


\section{Preliminaries} \label{sec:prelim}

In this section we collect some definitions that provide a more
general context for our main theorems and serve as ingredients for the
proofs.

\subsection{Iterated wreath products} First we elaborate on the
concept of an infinitely iterated wreath product.  Let
$\mS = (S_k)_{k\in \mathbb{N}}$ be a sequence of finite groups.  One
can define, in many ways, a new sequence of permutation groups
$\wh{\mS} = ( \wh{S}_k)_{k\in \mathbb{N}}$, with
$\wh{S}_k\le \Sym(\wh{m}_k)$, recursively as follows: (i) set
$\wh{S}_1 = S_1$ and choose a transitive faithful action of $\wh{S}_1$
on a finite set~$\wh{\Omega}_1$; (ii) for $k \ge 2$, let $\wh{S}_k$ be
the wreath product of $S_k$ by $\wh{S}_{k-1}$ w.r.t.\ the given
transitive faithful action of $\wh{S}_{k-1}$ and choose a transitive
faithful action of $\wh{S}_k$ on a finite set~$\wh{\Omega}_k$.  For
each $k \in \mathbb{N}$, the group $\wh{S}_k$ is called a
\emph{$k$-fold iterated wreath product of type $(S_1,\ldots,S_k)$}.
The resulting sequence $\wh{\mS}$ constitutes in a natural way an
inverse system of finite groups; we call its inverse limit
$\varprojlim \wh{S}_k$ an \emph{infinitely iterated wreath product of
  type $\mS$}.

The infinitely iterated wreath products w.r.t.\ product actions,
discussed in the introduction, fall into this scheme.  We are, in
fact, interested also in finitely iterated wreath products w.r.t.\
product actions and in iterated wreath products w.r.t.\ imprimitive
actions.  We employ the symbols $\cwr$ and $\wr$ to distinguish
between wreath products w.r.t.\ product actions and imprimitive
actions.  Using notation that is chosen to fit our later applications
(e.g., compare Proposition~\ref{prop:emb-wr-in-wr}), we describe the
two constructions as follows.

\begin{defn} \label{defn:two-constr} Let
  $\mS = (S_k)_{k\in \mathbb{N} \cup \{0\}}$ be a sequence of finite
  permutation groups $S_k \le \Sym(\Omega_k)$, and set
  $\mS' = (S_k)_{k\in \mathbb{N}}$.

  \smallskip

  (1) Define inductively $\wh{\Omega}_1 = \Omega_1$ and
  $\wh{\Omega}_n = \Omega_n \times \wh{\Omega}_{n-1}$ for $n \ge 2$.
  The \emph{$n$th iterated wreath product
    $W^\mathrm{ia}_n \le \Sym(\wh{\Omega}_n)$ of type
    $\mS'_n = (S_1,\ldots,S_n)$ w.r.t.\ imprimitive actions} is given
  by
  \begin{align*}
    W^\mathrm{ia}_1 & =  W^\mathrm{ia} (\mS'_1) = S_1\le \Sym(\wh{\Omega}_1), \\
    W^\mathrm{ia}_n & =  W^\mathrm{ia}(\mS'_n) =  S_n \wr W^\mathrm{ia}_{n-1} \le
                      \Sym(\wh{\Omega}_n) \qquad \text{for $n \ge 2$.}
  \end{align*}
  The explicit realisation of the wreath product as a semidirect
  product is recalled in the proof of Proposition~\ref{prop:P-embed}.
  The \emph{infinitely iterated wreath product of type $\mS'$ w.r.t.\
    imprimitive actions} is the inverse limit
  $W^\mathrm{ia}(\mS') = \varprojlim W^\mathrm{ia}_n$ of the natural
  inverse system
  $W^\mathrm{ia}_1 \twoheadleftarrow W^\mathrm{ia}_2 \twoheadleftarrow
  \ldots$.

  \smallskip

  (2) Define inductively $\wh{\Omega}_0 = \Omega_0$ and
  $\wh{\Omega}_n = \Omega_n^{\, \wh{\Omega}_{n-1}}$ for $n \ge 1$.
  The \emph{$n$th iterated wreath product
    $W^\mathrm{pa}_n \le \Sym(\wh{\Omega}_n)$ of type
    $\mS_n = (S_0, \ldots, S_{n-1})$ w.r.t.\ product actions} is given by
  \begin{align*}
    W^\mathrm{pa}_1 & = W^\mathrm{pa}(\mS_1) = S_0 \le \Sym(\wh{\Omega}_0), \\
    W^\mathrm{pa}_n & = W^\mathrm{pa} (\mS_n) = S_{n-1} \,\cwr\,
                      W^\mathrm{pa}_{n-1} \le  \Sym(\wh{\Omega}_{n-1})
                      \qquad \text{for $n \ge 2$}.  
  \end{align*}
  The explicit realisation of the wreath product as a semidirect
  product is recalled in the proof of Proposition~\ref{prop:P-embed}.
  The \emph{infinitely iterated wreath product of type $\mS$ w.r.t.\
    product actions} is the inverse limit
  $W^\mathrm{pa}(\mS) = \varprojlim W^\mathrm{pa}_n$ of the natural
  inverse system
  $W^\mathrm{pa}_1 \twoheadleftarrow W^\mathrm{pa}_2 \twoheadleftarrow
  \ldots$.
\end{defn}

\subsection{Composition series}
Let $G$ be a countably based profinite group.  Recall that every
descending sequence
$G = G_1 \prescript{}{\mathrm{o}}\ge\, G_2
\prescript{}{\mathrm{o}}\ge\, \ldots$
of open subgroups with $\bigcap_n G_n = 1$ forms a neighbourhood basis
of the identity element; see
\cite[Lemma~0.3.1(h)]{wilson:profinitegroups}.  A \emph{composition
  series} $(G_n)_{n\in \mathbb{N}}$ for $G$ consists of open subnormal
subgroups $G = G_1 \triangleright G_2 \triangleright \ldots$ with
$\bigcap_n G_n = 1$ and finite simple \emph{composition factors}
$S_n = G_n/G_{n+1}$ for $n \in \mathbb{N}$; we refer to
$\mS = (S_n)_{n\in \mathbb{N}}$ as a \emph{sequence of composition
  factors} for $G$.  The Jordan--H\"{o}lder Theorem for finite groups
implies that any two composition series of $G$ are equivalent in the
sense that the composition factors (up to isomorphism) occur with the
same multiplicities in both series.


\section{Embedding theorems}\label{sec:embedding}

In this section we establish the following basic fact which leads
directly to a proof of Theorem~\ref{thm:A}.

\begin{prop} \label{prop:emb-wr-in-wr}
  Let $\mS = (S_k)_{k \in \mathbb{N} \cup \{0\}}$ be a sequence of
  finite simple groups, where each $S_k \le \Sym(S_k)$ forms a
  permutation group via the right regular action.
  Then the infinitely iterated wreath product $W^\mathrm{ia}(\mS')$ of
  type $\mS' = ( S_k )_{k \in \mathbb{N}}$ w.r.t.\ imprimitive actions
  embeds as a closed subgroup into the infinitely iterated wreath
  product $W^\mathrm{pa}(\mS)$ of type
  $\mS = ( S_k )_{k \in \mathbb{N} \cup \{ 0\}}$ w.r.t.\ product
  actions.
\end{prop}

\begin{proof}[Proof of Theorem~\ref{thm:A}]
  Set $\mS' = (S_n )_{n\in \mathbb{N}}$, and let $G$ be a countably
  based profinite group that admits $\mS'$ as a sequence of
  composition factors.  Fix a composition series
  $G = G_1 \triangleright G_2 \triangleright \ldots$ with
  $S_n \cong G_n/G_{n+1}$ for $n \in \mathbb{N}$.
  For each $n \in \mathbb{N}$, choose representatives
  $T_n = \{t^{(n)}_{s_n} \in G_n \mid s_n \in S_n\}$ for the cosets of
  $G_{n+1}$ in $G_n$ and, for $g\in G_n$, denote by $[g]_n \in T_n$
  the representative of $g$ modulo~$G_{n+1}$.  The set
  \[
  \bigcup_{N \in \mathbb{N}_0} (T_N \times \cdots \times T_2 \times
  T_1),
  \] 
  of finite words in the `alphabet' $(T_n)_{n \in \mathbb{N}}$, forms
  a rooted spherically homogeneous tree $\mathcal{T}$ with respect to
  the prefix partial order, whose layers are in natural correspondence
  with the finite coset spaces $G_n \backslash G$.  As
  $\bigcap_n G_n = 1$, the group $G$ acts faithfully on the boundary
  $\partial \mathcal{T}$, and hence on $\mathcal{T}$, via right
  multiplication: for
  $(t^{(n)}_{s_n})_{n\in \mathbb{N}} \in \partial \mathcal{T}$ and
  $g\in G$ the element
  $(u_n)_{n\in \mathbb{N}} = ((t^{(n)}_{s_n})_{n\in \mathbb{N}})^g
  \in \partial \mathcal{T}$ is given recursively by
  \[
  g_1 = g, \quad u_n = [ t^{(n)}_{s_n} g_n]_n \text{ for $n \ge 1$,}
  \quad g_n = g_{n-1} u_{n-1}^{-1} \text{ for $n \ge 2$};
  \]
  compare \cite[proof of Theorem~2.6A]{dixon:permutationgroups}.

  This yields a continuous, hence closed embedding of the compact
  group $G$ into the profinite group $\mathrm{Aut}(\mathcal{T})$.
  Furthermore, by construction the image of $G$ lies in a subgroup
  $W \le \mathrm{Aut}(\mathcal{T})$ that is naturally isomorphic to
  $W^\mathrm{ia}(\mS')$.  Now, Proposition~\ref{prop:emb-wr-in-wr}
  shows that $W^\mathrm{ia}(\mS')$ and hence also $G$ embed as closed
  subgroups into $W^\mathrm{pa}(\mS)$.
\end{proof}

The proof of Proposition~\ref{prop:emb-wr-in-wr} relies on a
construction regarding finite wreath products.  For any set $X$ let
$\mathcal{P}(X) = \{ Y \mid Y \subseteq X \}$ denote the power set
of~$X$, and, for any given cardinal~$r$, we write
$\mathcal{P}_r(X) = \{ Y \in \mP(X) \mid \norm{Y}=r \}$.  A
permutation group $G \le \Sym(\Sigma)$ has a natural induced permutation
action on each $\mathcal{P}_r(\Sigma)$, via
$\Gamma^g = \{ \gamma^g \mid \gamma \in \Gamma \}$ for
$\Gamma \subseteq \Sigma$ and $g\in G$.

\begin{defn}\label{defn:Pembedded}
  Let $H\le \Sym(\Delta)$ and $G\le \Sym(\Sigma)$ be permutation groups.
  Consider the induced action of $G$ on $\mathcal{P}_r(\Sigma)$ for
  some cardinal~$r$.  We say that $H$ is \emph{P-embedded of degree
    $r$} in $G$ if there exist
  \begin{enumerate}
  \item[$\circ$] a collection
    $\wt{\Delta} \subseteq \mathcal{P}_r(\Sigma)$ of pairwise disjoint
    sets and
  \item[$\circ$] a subgroup $\wt{H} \le G$ such that $\wt{\Delta}$ is
    $\wt{H}$-invariant and the action of $\wt{H}$ on $\wt{\Delta}$ is
    equivalent to the action of $H$ on $\Delta$.
  \end{enumerate}
  In other words, $H\le \Sym(\Delta)$ is P-embedded of degree $r$ in
  $G \le \Sym(\Sigma)$ if there exist an isomorphism $\iota \colon H \to
  \wt{H} \le G$ and a bijection $\Gamma \colon \Delta \to \wt{\Delta}
  \subseteq \mP_r(\Sigma)$ such that
  \[
  \Gamma(\delta^h) = \Gamma(\delta)^{\iota(h)} \qquad \text{for $\delta
    \in \Delta$ and $h \in H$.}
  \]
  We remark that $H\le \Sym(\Delta)$ is P-embedded of degree $1$ in
  $G \le \Sym(\Sigma)$ if and only $H\le \Sym(\Delta)$ is
  \emph{permutationally isomorphic to a subgroup} of
  $G\le \Sym(\Sigma)$ as described in the introduction.
\end{defn}

\begin{prop} \label{prop:P-embed}
  Let $H \le \Sym(\Delta)$, $G \le \Sym(\Sigma)$ and $S \le \Sym(\Omega)$ be
  non-trivial finite permutation groups.  Suppose that 
  \begin{equation} \label{equ:iota-gamma} \iota \colon H \to \wt{H}
    \le G \qquad \text{and} \qquad \Gamma \colon \Delta \to
    \wt{\Delta} \subseteq \mP_r(\Sigma)
  \end{equation}
  provide a P-embedding of degree $r \ge 2$.
  
  Then the imprimitive wreath product
  $V = S \wr H \le \Sym(\Omega \times \Delta)$ can be P-embedded into
  the primitive wreath product
  $W = S \,\cwr\, G \le \Sym(\Omega^\Sigma)$.  More specifically,
  writing $\Phi = \Omega \times \Delta$, there are an integer
  $\wh{r}\ge 2$ and a P-embedding of degree $\wh{r}$ via
  \[
  \wh{\iota} \colon V \to \wt{V} \le W \qquad \text{and}
  \qquad \wh{\Gamma} \colon \Phi \to \wt{\Phi} \subseteq
  \mP_{\wh{r}}(\Omega^\Sigma),
  \]
  such that $\wh{\iota}$ induces, upon factoring out the base groups
  on both sides, the original isomorphism~$\iota$.
\end{prop}

\begin{proof}
  We identify $W = S \,\cwr\, G$ with $A \rtimes G$, where
  $A = S^\Sigma$ denotes the base group.  Elements
  $(s_\sigma)_\sigma \in A$ and $g \in G$ operate on $\Omega^\Sigma$
  by
  \[
  f^{(s_\sigma)_\sigma}(\tau) = f(\tau)^{s_\tau} \quad
  \text{and} \quad f^g(\tau) = f(\tau^{\,g^{-1}}) \qquad \text{for
    $f \in \Omega^\Sigma$, $\tau \in \Sigma$.}
  \]
   Similarly we identify
  $V = S \wr H$ with $B \rtimes H$, where $B = S^\Delta$ denotes the
  base group.  Elements $(s_\delta)_\delta \in B$ and $h \in H$
  operate on $\Omega \times \Delta$ by
  \[
  (\omega,\varepsilon)^{(s_\delta)_\delta} =
  (\omega^{s_\varepsilon},\varepsilon) \quad \text{and} \quad
  (\omega,\varepsilon)^h = (\omega,\varepsilon^h) \qquad \text{for
    $(\omega,\varepsilon) \in \Omega \times \Delta$.}
  \]
  The P-embedding~\eqref{equ:iota-gamma} yields a collection
  \[
  \wt{\Delta} = \{ \Gamma(\delta) \mid \delta \in \Delta \} \subseteq
  \mP_r(\Sigma)
  \]
  of pairwise disjoint $r$-element subsets of $\Sigma$ that are (i) in
  bijective correspondence with~$\Delta$ and (ii) being permuted by
  $\wt{H} \le G$ in the same way as the elements of $\Delta$ are being
  permuted by~$H$.
  
  We define $\wt{V} = \wt{B} \rtimes \wt{H} \le W$, where $\wt{B} \le
  A$ denotes the image of $B$ under the isomorphism
  \[
  \iota' \colon B \to \wt{B}, \quad (s_\delta)_{\delta \in \Delta} \mapsto
  (t_\sigma)_{\sigma \in \Sigma}, \quad \text{where }
  t_\sigma =
  \begin{cases}
    s_\delta & \text{if $\sigma \in \Gamma(\delta)$,} \\
    1 & \text{otherwise.}
  \end{cases}
  \]
  A routine verification shows that $\iota$ and $\iota'$ induce
  together an isomorphism $\wh{\iota} \colon V \to \wt{V}$ between
  groups.

  Recall that $\min(\norm{\Omega}, \norm{\Delta}, r) \ge 2$ and that
  we write $\Phi = \Omega \times \Delta$.  For 
  \[
  \wh{r} = \norm{\Omega}^{\norm{\Sigma} -
    r\norm{\Delta}} \cdot \big( \norm{\Omega}^{r} -
  \norm{\Omega} \big)^{\norm{\Delta} -1} \ge 2
  \]
  we obtain a bijection 
  \[
  \wh{\Gamma} \colon \Phi \to \wt{\Phi} = \big\{ \wh{\Gamma}(\varphi) \mid
  \varphi \in \Phi \big\} \subseteq \mP_{\wh{r}}(\Omega^\Sigma)
  \]  
  by setting, for each $\varphi = (\omega,\varepsilon) \in \Phi$,
  \begin{multline*}
    \wh{\Gamma}(\varphi) = \{ f \colon \Sigma \to \Omega \mid \text{$f$ is
      constant and equal to $\omega$ on $\Gamma(\varepsilon)$, but} \\
    \text{$f$ is not constant on any $\Gamma(\delta)$ for
      $\delta \in \Delta$ with $\delta \ne \varepsilon$} \}.
  \end{multline*}
  Moreover,
  $\wh{\Gamma}(\varphi) \cap \wh{\Gamma}(\varphi') = \varnothing$ for
  all $\varphi,\varphi' \in \Phi$ with $\varphi \ne \varphi'$.

  A routine calculation shows that, for
  $\varphi = (\omega,\varepsilon) \in \Phi$,
  \[
  \wh{\Gamma}(\varphi^{(s_\delta)_\delta}) =
  \wh{\Gamma}(\omega^{s_\varepsilon},\varepsilon) =
  \wh{\Gamma}(\omega,\varepsilon)^{\iota'((s_\delta)_\delta)} =
  \wh{\Gamma}(\varphi)^{\wh{\iota}((s_\delta)_\delta)} \quad \text{for
    $(s_\delta)_\delta \in B$}
  \]
  and
  \[
  \wh{\Gamma}(\varphi^h) = \wh{\Gamma}(\omega,\varepsilon^h) =
  \wh{\Gamma}(\omega,\varepsilon)^{\iota(h)} =
  \wh{\Gamma}(\varphi)^{\wh{\iota}(h)} \quad \text{for $h \in H$}.
  \]
  Thus $(\wh{\iota},\wh{\Gamma})$ provides the required P-embedding. 
\end{proof}

We obtain the following corollary which in turn supplies a proof of
Proposition~\ref{prop:emb-wr-in-wr}.

\begin{cor}\label{cor:finite-iterated-wrp}
  Let $\mS = (S_k)_{k\in \mathbb{N} \cup \{0\}}$, with
  $S_k \le \Sym(\Omega_k)$, be a sequence of non-trivial finite
  permutation groups, and set $\mS' = (S_k)_{k\in \mathbb{N}}$.

  \smallskip

  $\mathrm{(1)}$ For every $n \in \mathbb{N}$, the $n$th iterated
  wreath product $W^\mathrm{ia}(\mS'_n)$ of type
  $\mS'_n = (S_1,\ldots,S_n)$ w.r.t.\ imprimitive actions is
  P-embedded in the $(n+1)$th iterated wreath product
  $W^\mathrm{pa}(\mS_{n+1})$ of type $\mS_{n+1} = (S_0,\ldots,S_n)$
  w.r.t.\ product actions.

  \smallskip

  $\mathrm{(2)}$ The P-embeddings can be chosen compatible with one
  another so that they induce an embedding of $W^\mathrm{ia}(\mS')$
  into $W^\mathrm{pa}(\mS)$ as a closed subgroup.
\end{cor}
 

\section{Non-co-Hopfian iterated wreath products}\label{sec:selfsimilar}

In this section we prove Corollary~\ref{cor:B}, Theorem~\ref{thm:C}
and Corollary~\ref{cor:D}.  Recall that being ``permutationally
isomorphic to a subgroup'' of a permutation group is essentially the
same as being P-embedded of degree~$1$.

\begin{lem}\label{lem:cwrembedding}
  For $i \in \{1,2\}$ let $H_i \le \Sym(\Delta_i)$ and
  $G_i \le \Sym(\Omega_i)$ be non-trivial finite permutation groups, and
  suppose that
  \[ \iota_i \colon H_i \to \wt{H}_i
    \le G_i \qquad \text{and} \qquad \gamma_i \colon \Delta_i \to
    \wt{\Delta}_i \subseteq \Omega_i
  \]
  provide permutation isomorphisms of $H_i$ to subgroups of $G_i$.
  
  Then the primitive wreath product
  $V = H_1 \,\cwr\, H_2 \le \Sym(\Delta_1^{\, \Delta_2})$ is
  permutationally isomorphic to a subgroup of the primitive wreath
  product $W = G_1 \,\cwr\, G_2 \le \Sym(\Omega_1^{\, \Omega_2})$.
  More specifically, writing $\Phi = \Delta_1^{\, \Delta_2}$, there is
  a permutation isomorphism via
  \[
  \wh{\iota} \colon V \to \wt{V} \le W \qquad \text{and}
  \qquad \wh{\gamma} \colon \Phi \to \wt{\Phi} \subseteq
  \Omega_1^{\, \Omega_2},
  \]
  such that $\wh{\iota}$ induces, upon factoring out the base
  groups on both sides, the original
  isomorphism~$\iota_2 \colon H_2 \to \wt{H}_2$.
\end{lem}

\begin{proof}
  Similar to the proof of Proposition~\ref{prop:P-embed} we identify
  $W = G_1 \,\cwr\, G_2$ with $A \rtimes G_2$, where
  $A = G_1^{\, \Omega_2}$ denotes the base group, and
  $V = H_1 \,\cwr\, H_2$ with $B \rtimes H_2$, where
  $B = H_1^{\, \Delta_2}$ denotes the base group.  For notational
  simplicity we may assume, for $i \in \{1,2\}$, that $\gamma_i$ is
  just the identity map on $\Delta_i = \wt{\Delta}_i$.

  We define $\wt{V} = \wt{B} \rtimes \wt{H}_2 \le W$, where $\wt{B} \le
  A$ denotes the image of $B$ under the isomorphism
  $\iota' \colon B \to \wt{B}$ given by 
  \[
  (h_\delta)_{\delta \in \Delta_2} \mapsto (g_\omega)_{\omega \in
    \Omega_2}, \quad \text{where} \quad g_\omega =
  \begin{cases}
    h_\omega & \text{if $\omega \in \Delta_2$,} \\
    1 & \text{otherwise.}
  \end{cases}
  \]
  A routine verification shows that $\iota_2$ and $\iota'$ induce an
  isomorphism $\wh{\iota} \colon V \to \wt{V}$ of groups.

  Recall that we write $\Phi = \Delta_1^{\, \Delta_2}$, and fix an
  arbitrary point $\alpha \in \Omega_1$.  We obtain a bijection
  \[
  \wh{\gamma} \colon \Phi \to \wt{\Phi} = \{ \wh{\gamma}(f) \mid
  f \in \Phi \} \subseteq \Omega_1^{\, \Omega_2}
  \]  
  by setting, for each $f \colon \Delta_2 \to \Delta_1$ in $\Phi$,
  \[
  \wh{\gamma}(f) = \wt{f} \colon \Omega_2 \to \Omega_1, \quad
  \wt{f}(\omega) =
  \begin{cases}
    f(\omega) & \text{if $\omega \in \Delta_2$,} \\
    \alpha & \text{otherwise.}
  \end{cases}
  \]
  It is routine to verify that $(\wh{\iota},\wh{\gamma})$ gives
  the required permutation isomorphism.
\end{proof}

\begin{lem}\label{lem:cwrforgettingmiddle}
  Let $H \le \Sym(\Delta)$, $K \le \Sym(\Psi)$ and
  $G \le \Sym(\Omega)$ be non-trivial finite permutation groups.  Then
  the primitive wreath product
  $V = H \,\cwr\, G \le \Sym(\Delta^\Omega)$ is permutationally
  isomorphic to a subgroup of the primitive wreath product
  $W = H \,\cwr\, (K \,\cwr\, G) \le \Sym(\Delta^{(\Psi^\Omega)})$.
  More specifically, writing $\Phi = \Delta^\Omega$ and identifying
  the top groups of $V$ and $W$ as usual with~$G$, there is a
  permutation isomorphism via
  \[
  \iota \colon V \to \wt{V} \le W \qquad \text{and}
  \qquad \gamma \colon \Phi \to \wt{\Phi} \subseteq
  \Delta^{(\Psi^\Omega)}
  \]
  such that $\iota$ induces, upon factoring out the relevant base
  groups, the identity map between the top groups identified with~$G$.
\end{lem}

\begin{proof}
  We identify $V = H \,\cwr\, G$ with $B \rtimes G$, where
  $B = H^\Omega$ denotes the base group, and $U = K \,\cwr\, G$ with
  $C \rtimes \wt{G}$, where $C = K^\Omega$ denotes the base group and
  $\wt{G}$ is just a copy of~$G$.  We identify $W = H \,\cwr\, U$ with
  $A \rtimes (C \rtimes \wt{G})$, where $A = H^{\Psi^\Omega}$ denotes the base group. 
  
  Fix an element $\psi \in \Psi$.
  Setting, for $\omega \in \Omega$,
  \begin{multline*}
    \Gamma(\omega) = \big\{ f \colon \Omega \to \Psi \mid f(\omega)=\psi \text{ and $f$ is constant and } \\ \text{ different from $\psi$ on $\Omega\smallsetminus \{\omega\}$}  \big\}
  \end{multline*} 
  we obtain a P-embedding of degree $r =\norm{\Psi} -1$
  of $G \le \Sym(\Omega)$ into $U \le \Sym(\Psi^\Omega)$ via
  \[
  \iota' = \mathrm{id}_G \colon G \to \wt{G} \quad \text{and} \quad
  \Gamma \colon \Omega \to \wt{\Omega} = \{ \Gamma(\omega) \mid \omega
  \in \Omega \} \subseteq \mP_r(\Psi^\Omega).
  \]
  Next we define $\wt{V} = \wt{B} \rtimes \wt{G} \le W$, where $\wt{B} \le
  A$ denotes the image of $B$ under the isomorphism
  \[
  \iota'' \colon B \to \wt{B}, \quad (h_\omega)_{\omega \in \Omega} \mapsto
  (\wt{h}_f)_{f \in \Psi^\Omega}, \quad \text{where }
  \wt{h}_f =
  \begin{cases}
    h_\omega & \text{if $f \in \Gamma(\omega)$,} \\
    1 & \text{otherwise.}
  \end{cases}
  \]
  A routine verification shows that $\iota'$ and $\iota''$ induce an
  isomorphism $\iota \colon V \to \wt{V}$ of groups.

  Recall that we write $\Phi = \Delta^\Omega$, and fix an
  arbitrary point $\alpha \in \Delta$.  We obtain a bijection
  \[
  \gamma \colon \Phi \to \wt{\Phi} = \{ \gamma(F) \mid
  F \in \Phi \} \subseteq \Delta^{(\Psi^\Omega)}
  \]  
  by setting
  \[
  \gamma(F) = \wt{F} \colon \Psi^\Omega \to \Delta, \quad \wt{F}(f) =
  \begin{cases}
    F(\omega) & \text{for $f \in \Gamma(\omega)$,} \\
    \alpha & \text{otherwise.}
  \end{cases}
  \]
  It is routine to verify that $(\iota,\gamma)$ gives
  the required permutation isomorphism.
\end{proof}

\begin{prop} \label{prop:subseq-constr}
  Let $\mS = (S_k)_{k \in \mathbb{N}}$ be a sequence of non-trivial
  finite permutation groups $S_k \le \Sym(\Omega_k)$ and let
  $\mS^\circ = (S_{m(j)})_{j \in \mathbb{N}}$ for
  $m(1) < m(2) < \ldots$ be a subsequence of $\mS$.

  \begin{itemize}
  \item[$\mathrm{(1)}$] For every $n \in \mathbb{N}$, the $n$th
    iterated wreath product $W^{\mathrm{pa}}(\mS^\circ_n)$ of type
    $\mS^\circ_n = (S_{m(1)},\ldots,S_{m(n)})$ w.r.t.\ product actions
    is permutationally isomorphic to a subgroup of the $m(n)$th iterated
    wreath product $W^{\mathrm{pa}}(\mS_{m(n)})$ of type $\mS_{m(n)}$
    $= (S_1,\ldots,S_{m(n)})$ w.r.t.\ product actions.
  \item[$\mathrm{(2)}$] The permutation isomorphisms can be chosen
    compatible with one another so that they induce an embedding of
    $W^\mathrm{pa}(\mS^\circ)$ into $W^\mathrm{pa}(\mS)$ as a closed
    subgroup.
  \end{itemize}
\end{prop}

\begin{proof}
  We prove (1) by induction on $n \in \mathbb{N}$.

  For $n =1$, it suffices to observe that $S_{m(1)}$ is
  permutationally isomorphic to a subgroup of the primitive wreath
  product $S_{m(1)} \,\cwr\, W^{\mathrm{pa}}(\mS_{m(1)-1})$; e.g.,
  take the diagonal embedding of $S_{m(1)}$ into the base group acting
  on constant functions.

  Now suppose that $n \ge 2$.  By induction,
  $W^{\mathrm{pa}}(\mS^\circ_{n-1})$ is permutationally isomorphic to
  a subgroup of $W^{\mathrm{pa}}(\mS_{m(n-1)})$.  Repeated application
  of Lemma~\ref{lem:cwrforgettingmiddle} shows that
  $W^{\mathrm{pa}}(\mS^\circ_n) = S_{m(n)} \,\cwr\,
  W^{\mathrm{pa}}(\mS^\circ_{n-1})$
  is permutationally isomorphic to a subgroup of
  \[
  S_{m(n)} \,\cwr\, ( S_{m(n-1)+l} \,\cwr\,
  W^{\mathrm{pa}}(\mS_{m(n-1)+(l-1)})) = S_{m(n)} \,\cwr\,
  W^{\mathrm{pa}}(\mS_{m(n-1)+l})
  \]
  for $l \in \{1,\ldots,m(n)-m(n-1)-1\}$.  The final value for $l$
  yields the requested permutation isomorphism to $W^{\mathrm{pa}}(\mS_{m(n)})$.

  Claim (2) follows from the above construction and the
  compatibility assertion built into Lemma~\ref{lem:cwrforgettingmiddle}.
\end{proof}

The proofs of Corollary~\ref{cor:B} and Theorem~\ref{thm:C} are now
immediate. 

\begin{proof}[Proof of Corollary~\ref{cor:B}]
  Choose representatives $X_1,X_2, \ldots$ for the isomorphism types
  of finite simple groups in~$\mX$, and consider the sequence
  $\mS_\mX$ consisting of
  \[
  X_1, \,\, X_1,X_2, \,\, X_1,X_2,X_3, \,\, \ldots, \,\,
  X_1,X_2\ldots,X_n, \,\, \ldots.
  \]
  Then every countably based profinite group with composition factors
  in $\mX$ has a composition series that forms, up to isomorphisms, a
  subsequence of $\mS_\mX$.  Now apply Theorem~\ref{thm:A} and 
  Proposition~\ref{prop:subseq-constr}.
\end{proof}

\begin{proof}[Proof of Theorem~\ref{thm:C}]
  Suppose that the terms of $\mS = (S_k)_{k \in \mathbb{N}}$, with
  $S_k \le \Sym(\Omega_k)$, are eventually permutationally isomorphic to
  subgroups of later terms.  Choose $n_0 \in \mathbb{N}$ and a
  strictly increasing, but non-identity function
  $m \colon \mathbb{N}_{> n_0} \to \mathbb{N}_{> n_0}$ such that: for
  each $j \in \mathbb{N}_{> n_0}$, the permutation group
  $S_j \le \Sym(\Omega_j)$ is permutationally isomorphic to a subgroup of
  $S_{m(j)} \le \Sym(\Omega_{m(j)})$.

    Arguing similarly to the proof of Proposition~\ref{prop:subseq-constr}
  and applying, in addition, Lemma~\ref{lem:cwrembedding} in the
  induction step, we obtain an embedding of $W^\mathrm{pa}(\mS)$ as a
  proper closed subgroup into itself.
\end{proof}

Theorem~\ref{thm:C} highlights two natural questions. Do there exist
infinitely iterated wreath products of finite simple groups w.r.t.\
product actions that are co-Hopfian?  To what extent are the
hypotheses of Theorem~\ref{thm:C} irredundant?  Corollary~\ref{cor:D}
provides positive answers to both questions, when we restrict ourselves
to minimal finite non-abelian simple groups.

\begin{proof}[Proof of Corollary~\ref{cor:D}]
  In view of Theorem~\ref{thm:C} only one implication remains to be
  shown.  Suppose that $\mS$ consists of minimal finite non-abelian
  simple groups $S_k \le \Sym(\Omega_k)$, each equipped with a
  transitive permutation action and such that the terms of $\mS$ are
  not eventually permutationally isomorphic to a subgroup of a later
  terms.  This means that, for every $n_0 \in \mathbb{N}$, there
  exists $k \ge n_0$ such that $S_k \le \Sym(\Omega_k)$ is
  permutationally isomorphic to a subgroup of, and hence equivalent to
  $S_j \le \Sym(\Omega_j)$ for only finitely many $j \in \mathbb{N}$.
  Denote by $G = W^\mathrm{pa}(\mS)$ the infinitely iterated wreath
  product of type $\mS$ w.r.t.~product actions.  Suppose further that
  $H \le_\mathrm{c} G$ with $H\cong G$. We need to show that $H=G$.

  Since the action of each $S_j$ on $\Omega_j$ is transitive, it is
  easy to see that the open normal subgroups of $G$ form a descending
  chain $G = N_0 \supseteq N_1 \supseteq \ldots$, where 
  \[
  N_l = \mathrm{ker}(G\rightarrow W^{\mathrm{pa}}(\mS_l)) \quad \text{for
    $l\in \mathbb{N}$.}
  \]
  The group $H \cong G$  has a corresponding chain of open
  normal subgroups   $H = M_0 \supseteq M_1 \supseteq \ldots$.

  For every $n \in \mathbb{N}$ we choose $m(n) \in \mathbb{N}$ such
  that $H \cap N_n = M_{m(n)}$ and observe that
  $G/N_{m(n)} \cong H/M_{m(n)} \cong HN_n/N_n \le G/N_n$ implies
  $m(n) \le n$.  Moreover, it is enough to show that $m(n) \ge n$,
  hence $n=m(n)$, for infinitely many $n \in \mathbb{N}$; for this implies
  $H N_n = G$ for infinitely many $n$, and thus $H= G$.
   
  Now, start with any large number $n_0 \in \mathbb{N}$.  By our
  hypotheses, there is a $k \ge n_0$ such that
  \begin{multline*}
    n = \min \{ j \in \mathbb{N} \mid \text{$S_k \le \Sym(\Omega_k)$ is
      not equivalent to $S_l \le \Sym(\Omega_l)$} \\
    \text{for any $l>j$} \} \in \mathbb{N}_{\ge n_0}
  \end{multline*}
  is finite.  Clearly, the set of composition factors of
  $N_n$ is $\{S_j \mid j>n\}$. In particular, the group
  $N_n$ does not have any composition factors isomorphic to $S_k$.
  
  Set $m = m(n)$ and assume, for a contradiction, that $m<n$. Then
  $N_n \ge H \cap N_n = M_m \ge M_{n-1}$.  Observe that there exist
  $K \trianglelefteq_\mathrm{o} M_{n-1}$ such that $X = M_{n-1}/K$ is
  isomorphic to $S_k$.  We claim that $X$ is also a composition factor
  of $N_n$.  Indeed, by intersecting a composition series for $N_n$
  (possessing exclusively minimal non-abelian simple factors) with
  $M_{n-1}$ we obtain a subnormal series of $M_{n-1}$ with factors
  that are either soluble or isomorphic to a composition factor of
  $N_n$.  This implies that each composition factor of $M_{n-1}$ is
  either soluble or isomorphic to a composition factor of $N_n$.
  Consequently, $X \cong S_k$ is isomorphic to a composition factor of
  $N_n$, a contradiction.
\end{proof}

\begin{ackn*}
  Some of the results in this paper form part of the second author's
  PhD thesis, Royal Holloway University of London, 2015. The authors would also like to thank the referee for his careful reading and for making several
constructive suggestions regarding the exposition of this work.
\end{ackn*}


 
 \def\cprime{$'$}

\end{document}